\documentclass[preprint,12pt]{elsarticle}

\usepackage{amssymb}

\newtheorem{thm}{Theorem}

\newdefinition{dfn}{Definition}
\newdefinition{ex}{Example}
\newdefinition{conex}{Counterexample}
\newproof{proof}{Proof}

\newtheorem{rem}{Remark}

\begin{document}

\begin{frontmatter}



\title{$\alpha$-Parameterized Differential
Transform Method }

\author[]{K. Aydemir$^{a,}$}
\ead{kadriyeaydemr@gmail.com}
\author[]{O. Sh. Mukhtarov$^{a,b}$\corref{cor1}}
\ead{omukhtarov@yahoo.com}\cortext[cor1]{Corresponding Author (Tel:
+90 356 252 16 16, Fax: +90 356 252 15 85)}

\address[rvt]{Department of Mathematics, Faculty of Arts and Science, Gaziosmanpa\c{s}a University,\\
 60250 Tokat, Turkey}
 \address[]{Institute of Mathematics and Mechanics, Azerbaijan National Academy of Sciences,,\\
  Baku, Azerbaijan}


\begin{abstract}
In this paper we propose  a new version of differential transform
method  (we shall call this method as $\alpha$-parameterized
differential transform method), which differs from the traditional
differential transform method in calculating coefficients of Taylor
polynomials. Numerical examples are presented to illustrate the
efficiency and reliability of own method.  The result reveal that
$\alpha$-Parameterized differential transform method is a simple and
effective numerical algorithm.
\end{abstract}


\end{frontmatter}


\section{Introduction}
 Many problems
in mathematical physics, theoretical physics and chemical physics
are modelled by the so-called initial value and boundary value
problems in the second-order  ordinary differential equations. In
most cases, these problems may be too complicated to solve
analytically. Alternatively, the numerical methods can provide
approximate solutions rather than the analytic solutions of
problems. There are various approximation methods for solving a
system of differential equations, e.g. Adomian decomposition method
(ADM), Galerkin method, rationalized Haar functions method, homotopy
perturbation method (HPM), variational iteration method (VIM) and
the differential transform method (DTM).

The DTM is one of the numerical methods which enables to find
approximate solution in case of linear and non-linear systems of
differential equations. The main advantage of this method is that it
can be applied directly to nonlinear ODEs without requiring
linearization. The well known advantage of DTM is its simplicity and
accuracy in calculations and also wide range of applications.
Another important advantage is that this method is capable of
greatly reducing the size of computational work while still
accurately providing the series solution with fast convergence rate.
With this method, it is possible to obtain highly accurate results
or exact solutions for differential equations. The concept of the
differential transform method was first proposed by \cite{zh}, who
solved linear and nonlinear initial value problems in electric
circuit analysis. Afterwards, Chiou and Tzeng \cite{chi} applied the
Taylor transform to solve nonlinear vibration problems, Chen and Ho
\cite{ch} developed this method to various linear and nonlinear
problems such as two point boundary value problems and Ayaz
\cite{ay} applied it to the system of differential equations.
Abbasov et al. \cite{ab} used the method of differential transform
to obtain approximate solutions of the linear and non-linear
equations related to engineering problems and observed that the
numerical results are in good agreement with the analytical
solutions. In recent years  many authors has been used this method
for solving various types of equations. For example, this method has
been used for differential-algebraic equations \cite{co}, partial
differential equations \cite{ch, ja, ka, mo1}, fractional
differential equations \cite{mo} and difference equations \cite{ar}.
In \cite{so, so1, so2}, this method has been utilized for Telegraph,
Kuramoto-Sivashinsky and Kawahara equations. Shahmorad et al.
developed DTM to fractional-order integro-differential equations
with nonlocal boundary conditions \cite{na} and class of two
dimensional Volterra integral equations \cite{ta}. Abazari et. al.
applied this method for Schr\"{o}dinger equations \cite{bo}.
Different applications of DTM can be found in \cite{vd1,vd2}. Even
if the differential transform method (DTM) is an effective numerical
method for solving many initial value problem, there are also some
disadvantages, since this method  is designed for problems that have
analytic solutions (i.e. solutions that can be expanded in Taylor
series).

In this paper  we suggest a new version of DTM  which we shall
called  $\alpha$-parameterized differential transform method
($\alpha$-p DTM) to solve initial value and boundary value problems,
particularly eigenvalue problems. Note that, in the special cases
$\alpha=1 \ \textrm{and}  \ \alpha=0 $ of the $\alpha$-p DTM reduces
to the standard DTM.

\section{The classical DTM}
In this section, we  describe  the definition and some basic
properties of
 the classical differential transform method. An arbitrary analytic function  f(x) can
be expanded in Taylor series about a point $x=x_0$ as
\begin{eqnarray}\label{dtm1}
f(x)=\sum\limits_{k=0}^{\infty}\frac{(x-x_0)^{k}}{k!}\left[\frac{d^{k}f(x)}{dx^{k}}
\right]_{x=x_0}
\end{eqnarray}
The classical differential transformation of f(x) is defined as
\begin{eqnarray}\label{dtm2}
F(k)=\frac{1}{k!}\left[\frac{d^{k}f(x)}{dx^{k}} \right]_{x=x_0}
\end{eqnarray}
Then the inverse differential transform is
\begin{eqnarray}\label{dtm3}
f(x)=\sum_{k=0}^{\infty}(x-x_0)^{k}F(k).
\end{eqnarray}
The fundamental mathematical operations performed by differential
transform method are listed in following\\
i) If $f(x)=g(x)\pm h(x)$ then $F(k)=G(k)\pm H(k)$\\
ii) If $f(x)=\alpha g(x)$, $\alpha \in \mathbb{R}$, then
$F(k)=\alpha G(k)$\\
iii) If $f(x)=\frac{d^{m}g}{dx^{m}}$ then
$F(k)=(k+1)(k+2)...(k+m)G(k+m)$\\
iv) If $f(x)=x^{m}$ then $F(k)=\delta(k-m)= \left\{
\begin{array}{c}
1 \ \  for \ k=m  \\
0 \ \  for \ k\neq m
\end{array}
\right.$\\
v) If $f(x)=g(x) h(x)$ then $F(k)=\sum\limits_{m=0}^{k}H(m) G(k-m)$\\
\section{$\alpha$-Parameterized differential
transform method ($\alpha$-p DTM)} In this section we suggest a new
version of classical differential transform method
  by
following.

\label{dtm4} Let $I=[a,b]\subset \mathbb{R}$ be an arbitrary real
interval, $f:I\rightarrow\mathbb{R}$ is an infinitely differentiable
function (in real applications it is enough to required that $f(x)$
is sufficiently large order differentiable function) , $\alpha \in
[0,1]$ any real parameter and
 N any integer (large enough). Let us introduce the following
 notations
 \begin{eqnarray}\label{dtmm1}
D_{a}(f;k):=\frac{ f^{(k)}(a)}{k!}, \ \ \ D_{b}(f;k):= \frac{
f^{(k)}(b)}{k!}
\end{eqnarray}
\begin{eqnarray}\label{dtmm25}
D(f,\alpha;k):=\alpha D_{a}(f;k) +(1-\alpha) D_{b}(f;k), \ \ \ \
\alpha\in [0,1], \ \ k\in \mathbb{N}
\end{eqnarray}
\begin{dfn}\label{dfnn}
The sequence
 $$(D_{\alpha}(f)):=(D(f,\alpha;1), D(f,\alpha;2),...)$$
is called the $\alpha$-P transformation of the original function
f(x). The so-called ,,differential inverse" transform of
$D_{\alpha}(f)$ we define as
\begin{eqnarray}\label{dtmm6}
E_{\alpha}(D_{\alpha}(f)):=\sum_{k=0}^{\infty}D(f,\alpha;k)(x-x_{\alpha})^{k}
\end{eqnarray}
if the series is convergent, where $x_{\alpha}=\alpha
a+(1-\alpha)b$.
\end{dfn}
 The
function $\widetilde{f_{\alpha}}(x)$ defined by equality
$$\widetilde{f_{\alpha}}(x):=E_{\alpha}(D_{\alpha}(f))$$
we called the $\alpha$-parameterized approximation of the function
f(x).
\begin{rem}\label{rem1} In the cases of $\alpha=1 \  \textrm{and}  \ \alpha=0 $ the $\alpha$-p differential
transform (\ref{dtmm25}) reduces to the classical differential
transform (\ref{dtm2}) at the points  $x=a \  \textrm{and}  \ x=b $
respectively. Namely for $\alpha=0 \  \textrm{and}  \ \alpha=1 $ the
equality $\widetilde{f_{\alpha}}(x)=f(x)$ is hold.
\end{rem}
\begin{rem}\label{remn1}
For  practical application, instead of $\widetilde{f_{\alpha}}(x)$
it is convenient to introduced $N$-term $\alpha$-parameterized
approximation of the function $\widetilde{f_{\alpha}}(x)$ which we
shall define as
\begin{eqnarray}\label{dtmmn6}
\widetilde{f}_{\alpha,N}(x):=E_{\alpha,N}(D_{\alpha}(f)):=\sum_{k=0}^{N}D(f,\alpha;k)(x-x_{\alpha})^{k}
\end{eqnarray}
\end{rem}
\begin{thm}\label{rem11} If f(x) is constant function then
$\widetilde{f_{\alpha}}(x)=f(x)$ and
$\widetilde{f}_{\alpha,N}(x)=f(x)$ for each $N$.
\end{thm}
\begin{proof}The proof is immediate from Definition \ref{dfnn} and
Remark \ref{remn1}.
\end{proof}
\begin{thm}\label{dtm10}
 If $f(x)=c g(x)$, $c \in \mathbb{R}$, then
$D_{\alpha}(f)=c D_{\alpha}(g)  \ \textrm{and} \
\widetilde{f_{\alpha}}(x)=c \widetilde{g_{\alpha}}(x).$
\end{thm}
\begin{proof}\label{dtm11}
By applying the well-known properties of classical DTM we get
 \begin{eqnarray}\label{dtm12}
D(f,\alpha;k)&=&\alpha D_{a}(f;k) +(1-\alpha)
D_{b}(f;k)\nonumber\\&=&\alpha c  D_{a}(g;k)+(1-\alpha)c D_{b}(g;k)
\nonumber\\&=&c(\alpha   D_{a}(g;k)+(1-\alpha)
D_{b}(g;k))\nonumber\\&=&c D(g,\alpha;k)
\end{eqnarray}
Consequently $D_{\alpha}(f)=c D_{\alpha}(g)$, from which immediately
follows that $\widetilde{f_{\alpha}}(x)=c
\widetilde{g_{\alpha}}(x)$.
\end{proof}
\begin{thm}\label{dtm8}
 If $f(x)=g(x)\pm h(x)$ then $D_{\alpha}(f)=D_{\alpha}(g)\pm
D_{\alpha}(h) \ \textrm{and} \ \widetilde{f_{\alpha}}(x)=
\widetilde{g_{\alpha}}(x)\pm \widetilde{h}_{\alpha}(x).$
\end{thm}
\begin{proof} By using the definition of transform (\ref{dtmm25})
 \begin{eqnarray}\label{dtm9}
D(f,\alpha;k)&=&\alpha D_{a}(f;k)\pm(1-\alpha)
D_{b}(f;k)\nonumber\\&=&\alpha D_{a}(g+h;k)\pm(1-\alpha)
D_{b}(g+h;k)\nonumber\\&=& D(g,\alpha;k)\pm D(h,\alpha;k)
\end{eqnarray}
Consequently $D_{\alpha}(f)=D_{\alpha}(g)\pm D_{\alpha}(h)$, from
which immediately follows that $ \widetilde{f_{\alpha}}(x)=
\widetilde{g}_{\alpha}(x)\pm \widetilde{h}_{\alpha}(x)$.
\end{proof}

\begin{thm}\label{dtm13}
Let  $f(x)=\frac{d^{m}g}{dx^{m}}$ and $m\in\mathbb{N}$. Then
$$D(f^{(m)},\alpha;k)=\frac{(k+m)!}{k!}D(f,\alpha;k+m)$$ and
$$\widetilde{f}_{\alpha}^{(m)}(x)=\sum\limits_{k=0}^{\infty}\frac{(k+m)!}{k!}D(f,\alpha;k+m)(x-x_{\alpha})^{k}$$
where  $x_{\alpha}=\alpha a+(1-\alpha)b$.
\end{thm}
\begin{proof}\label{dtm14}We have from definition (\ref{dtmm25})
 \begin{eqnarray*}\label{dtm15}
D(f^{(m)},\alpha;k)&=&\alpha D_{a}(f^{(m)};k)+(1-\alpha)
D_{b}(f^{(m)};k)\nonumber\\&=&\alpha (k+1)(k+2)...(k+m)
D_{a}(f;k+m)\nonumber\\&+&(1-\alpha)(k+1)(k+2)...(k+m)
D_{b}(f;k+m)\nonumber\\&=&(k+1)(k+2)...(k+m)(\alpha
D_{a}(f;k)+(1-\alpha)D_{b}(f;k+m))\nonumber\\&=&\frac{(k+m)!}{k!}D(f,\alpha;k+m)
\end{eqnarray*}
Thus we get $D(f^{(m)},\alpha;k)=\frac{(k+m)!}{k!}D(f,\alpha;k+m)$.
Using this we find
$\widetilde{f}_{\alpha}^{(m)}(x)=\sum\limits_{k=0}^{\infty}\frac{(k+m)!}{k!}D(f,\alpha;k+m)(x-x_{\alpha})^{k}$.
\end{proof}
\begin{thm}\label{dtm16}
 Let $f(x)=x^{m}$, $m\in\mathbb{N}$. Then $$D(f,\alpha;k)=\left\{
\begin{array}{c}
\left(%
\begin{array}{c}
  m \\
  k \\
\end{array}%
\right)(\alpha a^{m-k}+(1-\alpha)b^{m-k}) \ \ for  \ k<m \\
 1 \ \ \ \ for \ \ \  \
k=m\\
0 \ \  \ \ for \  \ k> m
\end{array}
\right.$$
\end{thm}
\begin{proof}\label{dtm17} Let $ k<m$. By using the definition of the transform
(\ref{dtmm25}) we have
 \begin{eqnarray}\label{dtm18}
D(f,\alpha;k)&=&\alpha D_{a}(f;k)+(1-\alpha)
D_{b}(f;k)\nonumber\\&=&\alpha D_{a}(x^{(m)};k)+(1-\alpha)
D_{b}(x^{(m)};k)\nonumber\\&=&\left(%
\begin{array}{c}
  m \\
  k \\
\end{array}%
\right)(\alpha a^{(m-k)}+(1-\alpha)b^{(m-k)})
\end{eqnarray}
 The equalities $\ D(x^{m},\alpha;m)=1 $ and $\
D(x^{m},\alpha;m+s)=0$ for $s\geq1$ is obvious.
\end{proof}
\begin{thm}\label{dtm16}
If $f(x)=g(x) h(x)$ then $D(f,\alpha;k)=\sum\limits_{m=0}^{k}[\alpha
 D_{a}(g;m)D_{a}(h;k-m)+(1-\alpha)D_{b}(g;m)D_{b}(h;k-m)]
$
\end{thm}
\begin{proof}\label{dtm17}By using the definition of transform given in Eq.
(\ref{dtmm25}) we have
 \begin{eqnarray}\label{dtm18}
D(f,\alpha;k)&=&\alpha D_{a}(f;k)+m(1-\alpha)
D_{b}(f;k)\nonumber\\&=&\alpha D_{a}(gh;k)+(1-\alpha)
D_{b}(gh;k)\nonumber\\&=&\frac{\alpha}{k!}\sum\limits_{m=0}^{k}
 \left(%
\begin{array}{c}
  k \\
  m \\
\end{array}%
\right)g^{(m)}(a)h^{(k-m)}(a)+\frac{(1-\alpha)}{k!}\sum\limits_{m=0}^{k}
 \left(%
\begin{array}{c}
  k \\
  m \\
\end{array}%
\right)g^{(m)}(b)h^{(k-m)}(b)\nonumber\\&=&\sum\limits_{m=0}^{k}[\alpha
 D_{a}(g;m)D_{a}(h;k-m)+(1-\alpha)D_{b}(g;m)D_{b}(h;k-m)].
\end{eqnarray}
\end{proof}

\section{Justification  of the  $\alpha$-p DTM  }In order to show the effectiveness of $\alpha$-p DTM for solving  boundary value problems
, examples is demonstrated.
\begin{ex}\label{dtmm16}(Application to boundary-value problem) Let us consider the differential equation
\begin{eqnarray}\label{dtmm17}
\ell y:= y''(x)+\mu^{2}y(x)=0, \ \ \ \ x \in [0,1],  \ \ \mu
 \in \mathbb{R}.
\end{eqnarray}
with the boundary conditions
 \begin{eqnarray}\label{dtmnm18}
 y(0)=0, \ \ \  y(1)=1.
\end{eqnarray}
Exact solution for this problem is
\begin{eqnarray}
y(x) =\frac{\sin\mu x}{\sin\mu}.
\end{eqnarray}
\end{ex}
Applying the $N$-term $\alpha$-p differential transform  of both
sides (\ref{dtmm17}) and (\ref{dtmnm18}) we obtain the following
$\alpha$-parameterized boundary value problem as
\begin{eqnarray}\label{dý17}
(\widetilde{\ell y})_{\alpha,N}=\widetilde{0}_{\alpha,N}, \ \
\widetilde{y}_{\alpha,N}(0)=\widetilde{0}_{\alpha,N}, \ \ \
\widetilde{y}_{\alpha;N}(1)=\widetilde{1}_{\alpha,N}.
\end{eqnarray}
By using the fundamental operations of $\alpha$-p DTM we have
\begin{eqnarray}\label{or1}
D(y,\alpha;k+2)=-\frac{\mu^{2} D(y,\alpha;k)}{ (k+1)(k+2)}
\end{eqnarray}
The boundary conditions given in (\ref{dtmnm18}) can be transformed
as follows
\begin{eqnarray}\label{or2}
\widetilde{y}_{\alpha,N}(0)=\sum_{k=0}^{N}D(y,\alpha;k)(\alpha-1)^{k}=0
\ \ \ \textrm{and} \ \ \
\widetilde{y}_{\alpha,N}(1)=\sum_{k=0}^{N}D(y,\alpha;k)\alpha^{k}=1
\end{eqnarray}
Using (\ref{or1}) and (\ref{or2}) and by taking $N=5$, the following
$\alpha$-p approximate solution is obtained
\begin{eqnarray}\label{or3}
\widetilde{y}_{\alpha}(x)&=&A+(x-x_\alpha)B-\frac{\mu^{2}(x-x_\alpha)^{2}A}{2}-\frac{\mu^{2}(x-x_\alpha)^{3}B}{6}+\frac{\mu^{4}(x-x_\alpha)^{4}A}{24}
\nonumber\\&+&\frac{\mu^{4}(x-x_\alpha)^{5}B}{120}+O(x^6)
\end{eqnarray}
where $x_\alpha=(1-\alpha)$, according to (\ref{dtmm6}),
$D(y,\alpha;0)=A$ and $D(y,\alpha;1)=B.$ The constants $A$ and $B$
evaluated from equations in (\ref{or1}) as follow
\begin{eqnarray}\label{eq}
A&=&(2880(x_{0}^{4}\mu^{8}-4x_{0}^{5}\mu^{8}+6x_{0}
^{6}\mu^{8}+24\mu^{4}-4x_{0} ^{7}\mu^{8}-480\mu^{2}+x_{0} ^{8}\mu^{8}+2880 \nonumber\\
&-& 12x_{0} ^{2}\mu^{6}+40x_{0} ^{3}\mu^{6} -60x_{0} ^{4}\mu^{6}
+48x_{0} ^{5}\mu^{6} -16x_{0} ^{6}\mu^{6})^{-1} )\nonumber\\
&\times&(x_{0} -\frac{%
\mu^{2} x_{0} ^{3}}{6}+\frac{\mu^{4}x_{0} ^{5}}{120})
\end{eqnarray}
and
\begin{eqnarray}\label{eq1}
 B&=&(2880(x_{0}^{4}\mu^{8}-4x_{0}^{5}\mu^{8}+6x_{0}
^{6}\mu^{8}+24\mu^{4}-4x_{0} ^{7}\mu^{8}-480\mu^{2}+x_{0} ^{8}\mu^{8}+2880 \nonumber\\
&-& 12x_{0} ^{2}\mu^{6}+40x_{0} ^{3}\mu^{6} -60x_{0} ^{4}\mu^{6}
+48x_{0} ^{5}\mu^{6} -16x_{0} ^{6}\mu^{6})^{-1} )\nonumber\\
&\times&(1-\frac{\mu^{2} x_{0} ^{2}}{%
2 }+\frac{\mu ^{4}x_{0}^{4}}{24})
\end{eqnarray}
\begin{ex}\label{dtmmm16}(Application to eigenvalue problems)We consider the following eigenvalue problem
\begin{eqnarray}\label{dtmmm17}
 y''+\lambda y=0, \ \ \ \ x \in [0,1]
\end{eqnarray}
 \begin{eqnarray}\label{dtmmb18}
A_{11} y(0)+A_{12} y'(0)=0 \\ \label{dtmmmb18}
 A_{21} y(1)+A_{22} y'(1)=0
\end{eqnarray}
Taking the $\alpha$-p differential transform  of both sides
(\ref{dtmmm17}) we find
\begin{eqnarray}\label{dýn17}
D(y''+\lambda y,\alpha;k)=(k+1)(k+2)D(y,\alpha;k+2)+\lambda
D(y,\alpha;k)=0.
\end{eqnarray}
Thus the following recurrence relation is obtained
\begin{eqnarray}\label{onr1}
D(y,\alpha;k+2)=-\frac{\lambda D(y,\alpha;k)}{ (k+1)(k+2)}.
\end{eqnarray}
Using definition of  $\alpha$-p differential transform we get
\begin{eqnarray}\label{oyrn2}
\widetilde{y}_{\alpha}(x)=\sum_{k=0}^{\infty}D(y,\alpha;k)(x-x_{\alpha})^{k}
\\ \label{ogrn2} \widetilde{y'}_{\alpha}(x)=\sum_{k=0}^{\infty}k
D(y,\alpha;k)(x-x_{\alpha})^{k-1}
\end{eqnarray}
Consuquently
\begin{eqnarray}\label{orrn2}
\widetilde{y}_{\alpha}(0)=\sum_{k=0}^{\infty}D(y,\alpha;k)(\alpha-1)^{k}=\sum_{k=0}^{\infty}(-1)^{k}D(y,\alpha;k)(1-\alpha)^{k}
\\ \label{orrsn2} \widetilde{y'}_{\alpha}(0)=\sum_{k=0}^{\infty}k
D(y,\alpha;k)(\alpha-1)^{k-1}=\sum_{k=0}^{\infty}(-1)^{k}k
D(y,\alpha;k)(1-\alpha)^{k-1}
\end{eqnarray}
Thus the boundary condition given in (\ref{dtmmb18}) can be
transformed as follows
\begin{eqnarray}\label{dtsmmmb18}
 A_{11} \widetilde{y}_{\alpha}(0)+A_{12} \widetilde{y'}_{\alpha}(0)
=\sum_{k=0}^{\infty}( A_{11}(\alpha-1)^{k}+k
A_{12}(\alpha-1)^{k-1})D(y,\alpha;k)=0
\end{eqnarray}
Similarly we have
\begin{eqnarray}\label{ors2}
\widetilde{y}_{\alpha}(1)=\sum_{k=0}^{\infty}D(y,\alpha;k)\alpha^{k}
\end{eqnarray}
and
\begin{eqnarray}\label{on2}
\widetilde{y'}_{\alpha}(1)=\sum_{k=0}^{\infty}k
D(y,\alpha;k)\alpha^{k-1}
\end{eqnarray}
In this case the boundary condition given in (\ref{dtmmmb18}) can be
written as follows
\begin{eqnarray}\label{mnk}
 A_{21} \widetilde{y}_{\alpha}(1)+A_{22} \widetilde{y'}_{\alpha}(1)
=\sum_{k=0}^{\infty}( A_{21}\alpha^{k}+k
A_{22}\alpha^{k-1})D(y,\alpha;k)=0
\end{eqnarray}
Let $D(y,\alpha;0)=A \ \ \ \textrm{and} \ \ \ D(y,\alpha;1)=B.$
Substituting these values in (\ref{onr1}) we have the following
recursive procedure
\begin{eqnarray}\label{mnnk}
D(y,\alpha;k)=\left\{
                \begin{array}{ll}
                  \frac{A(-\lambda)^{\ell}}{(2\ell!)},   \ \textrm{ for} \  k=2\ell \\
                  \frac{B(-\lambda)^{\ell}}{(2\ell+1)!},   \ \textrm{for} \    k=2\ell+1
                \end{array}
              \right.
\end{eqnarray}Substituting (\ref{mnnk}) in (\ref{dtsmmmb18})and (\ref{mnk}) we find
\begin{eqnarray}\label{dn8}
&&A\{\sum_{\ell=0}^{\infty}( A_{11}(\alpha-1)^{2\ell}+2\ell
A_{12}(\alpha-1)^{2\ell-1})\frac{(-\lambda)^{\ell}}{(2\ell)!}\}\nonumber\\&+&B\{\sum_{\ell=0}^{\infty}(
A_{11}(\alpha-1)^{2\ell+1}+(2\ell+1)
A_{12}(\alpha-1)^{2\ell})\frac{(-\lambda)^{\ell}}{(2\ell+1)!}\}=0
\end{eqnarray}and
\begin{eqnarray}\label{dmn8}
&& A\{\sum_{\ell=0}^{\infty}( A_{21}\alpha^{2\ell}+2\ell
A_{22}\alpha^{2\ell-1})\frac{(-\lambda)^{\ell}}{(2\ell)!}\}\nonumber\\&+&B\{\sum_{\ell=0}^{\infty}(
A_{21}\alpha^{2\ell+1}+(2\ell+1)
A_{22}\alpha^{2\ell})\frac{(-\lambda)^{\ell}}{(2\ell+1)!}\}=0.
\end{eqnarray}
respectively. In this case we have a linear system  of the equations
with respect to the variables $A$ and $B$ as
\begin{eqnarray}\label{dnnn8}
AP_{11}(\lambda)+B P_{12}(\lambda)=0 \\ \label{dmmnn8} A
P_{21}(\lambda)+B P_{22}(\lambda)=0
\end{eqnarray}
where $P_{11}(\lambda):=\sum_{\ell=0}^{\infty}(
A_{11}(\alpha-1)^{2\ell}+2\ell
A_{12}(\alpha-1)^{2\ell-1})\frac{(-\lambda)^{\ell}}{(2\ell)!},$
$P_{12}(\lambda):=\sum_{\ell=0}^{\infty}(
A_{11}(\alpha-1)^{2\ell+1}+(2\ell+1)
A_{12}(\alpha-1)^{2\ell})\frac{(-\lambda)^{\ell}}{(2\ell+1)!}$,
$P_{21}(\lambda):=\sum_{\ell=0}^{\infty}( A_{21}\alpha^{2\ell}+2\ell
A_{22}\alpha^{2\ell-1})\frac{(-\lambda)^{\ell}}{(2\ell)!}$ and
$P_{22}(\lambda):=\sum_{\ell=0}^{\infty}(
A_{21}\alpha^{2\ell+1}+(2\ell+1)
A_{22}\alpha^{2\ell})\frac{(-\lambda)^{\ell}}{(2\ell+1)!}$. Since
the system (\ref{dnnn8})-(\ref{dmmnn8}) has a nontrivial solution
for $A$ and $B$ the characteristic determinant is zero i.e.
$$P(\lambda)=\left|
    \begin{array}{cc}
      P_{11}(\lambda) & P_{12}(\lambda) \\
      P_{21}(\lambda) & P_{22}(\lambda) \\
    \end{array}
  \right|=0.
$$
Thus we have characteristic equation for eigenvalues by $\alpha$-p
DTM Now, let us find the exact eigenvalues and eigenfunctions of the
Sturm-Liouville problem (\ref{dtmmm17})-(\ref{dtmmmb18}). The
general solution of equation (\ref{dtmmmb18}) have the form
\begin{eqnarray}\label{gen}
y(x)=C\cos \mu x+ D\sin \mu x
\end{eqnarray}
where $\lambda=\mu^{2}$ and $C, D$ are arbitrary constants. Applying
the boundary conditions (\ref{dtmmm17}), (\ref{dtmmb18})  we get
 \begin{eqnarray}\label{dff}
&&A_{11} C+\mu A_{12} D=0 \\ \label{dfff} && (A_{21}\cos -\mu
A_{22}\sin \mu)C+(A_{21}\cos -\mu A_{22}\sin \mu)D=0
\end{eqnarray}
Because we cannot have $C=D=0$, this implies
 \begin{eqnarray}\label{dfm}
(A_{11}A_{21}+\mu^{2} A_{12}A_{22})\sin \mu-\mu(A_{12}A_{21}-
A_{11}A_{22})\cos \mu=0.
\end{eqnarray}
This is transcendental equation which is solved graphically. Let
$\mu=\mu_n, n\in \mathbb{N}$ are points of intersection of the
graphs of the functions
 \begin{eqnarray}\label{kl}
&&y=(A_{11}A_{21}+\mu^{2} A_{12}A_{22})\sin \mu \ \ \  \textrm{and} \\
\label{kll} &&y=\mu(A_{12}A_{21}- A_{11}A_{22})\cos \mu
\end{eqnarray}
The eigenvalues and corresponding eigenfunctions are therefore given
by
\begin{eqnarray}\label{kmm}
\lambda_n=\mu_{n}^{2} \ \   \textrm{and} \ \  y_n(x)=C_n\cos \mu_n
x+ D_n\sin \mu_n x, \ n\in \mathbb{N}.
\end{eqnarray}
Now we consider special case of the Sturm-Liouville problem
(\ref{dtmmmb18})-(\ref{dtmmm17}),
\begin{eqnarray}\label{sp}
&&y''+\lambda y=0 \\
  \label{spl}&&y(0)+y'(0)=0 \\ \label{spll}&&y(1)-y'(1)=0.
\end{eqnarray}
The eigenvalues of this problem are determined by the equation
\begin{eqnarray}\label{spp}
\tan \mu=\frac{2\mu}{1-\mu^{2}}.
\end{eqnarray}
This equation can be solved graphically by the points of
intersections of the graphs of functions
\begin{eqnarray}\label{stpp}
y=\tan \mu \ \ \textrm{and} \ \ y=\frac{2\mu}{1-\mu^{2}}
\end{eqnarray}
as shown by the sequence $(\mu_n)$ in Figure 1.

\begin{figure}
  \includegraphics[height=6cm,width=12cm]{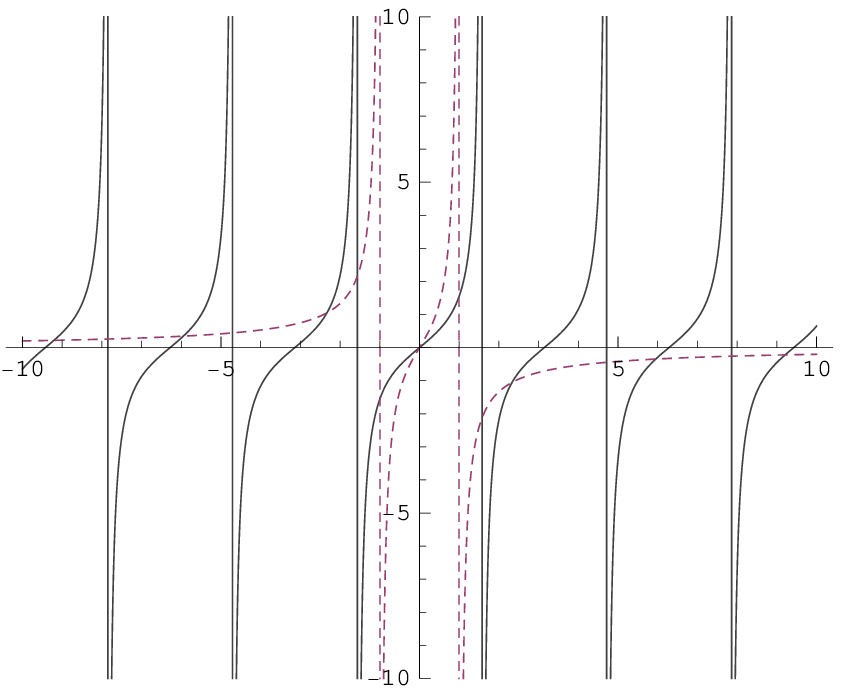}
  \caption{ }\label{f1}
\end{figure}

The eigenvalues of the considered problem are given by
$\lambda_n=\mu_{n}^{2}$ and corresponding eigenfunctions are given
by
\begin{eqnarray}\label{sstpp}
y_n(x)=C_n\cos \mu_n x+ D_n\sin \mu_n x, \ n\in \mathbb{N}.
\end{eqnarray}
Taking the $\alpha$-p differential transform  of both sides the
equation (\ref{sp}) the following recurrence relation is obtained
\begin{eqnarray}\label{oddh}
D(y,\alpha;k+2)=-\frac{\lambda D(y,\alpha;k)}{ (k+1)(k+2)}.
\end{eqnarray}

Applying N-term $\alpha$-p differential transform  the boundary
conditions  (\ref{spl})-(\ref{spll}) are transformed as follows:
\begin{eqnarray}\label{kml}
 && \widetilde{y}_{\alpha}(0)+ \widetilde{y'}_{\alpha}(0)
=\sum_{k=0}^{N}((\alpha-1)^{k}+k (\alpha-1)^{k-1})D(y,\alpha;k)=0
\\ \label{khg}
 &&\widetilde{y}_{\alpha}(1)-\widetilde{y'}_{\alpha}(1)
=\sum_{k=0}^{N}( \alpha^{k}-k\alpha^{k-1})D(y,\alpha;k)=0
\end{eqnarray}
By using (\ref{oddh}), (\ref{kml}) and (\ref{khg}) we obtain the
following equalities (for N=6)
\begin{eqnarray}\label{dnny8}
&&
A\bigg[1+((\alpha-1)^{2}+2(\alpha-1))\frac{(-\lambda)}{2!}+((\alpha-1)^{4}+4(\alpha-1)^{3})\frac{\lambda^{2}}{4!}\nonumber\\&+&
((\alpha-1)^{6}+6(\alpha-1)^{5})\frac{(-\lambda^{3})}{6!}\bigg]+B\bigg[\alpha+((\alpha-1)^{3}+3(\alpha-1)^{2})\frac{(-\lambda)}{3!}\nonumber\\
&+&((\alpha-1)^{5}+5(\alpha-1)^{4})\frac{\lambda^{2}}{5!}\bigg]=0
\end{eqnarray}and
\begin{eqnarray}\label{dnyy8}
&&
A\bigg[1+(\alpha^{2}-2\alpha)\frac{(-\lambda)}{2!}+(\alpha^{4}-4\alpha^{3})\frac{\lambda^{2}}{4!}+
(\alpha^{6}-6\alpha^{5})\frac{(-\lambda^{3})}{6!}\bigg]\nonumber\\&+&B\bigg[(\alpha-1)+(\alpha^{3}-3\alpha^{2})\frac{(-\lambda)}{3!}+(\alpha^{5}-5\alpha^{4})\frac{\lambda^{2}}{5!}\bigg]=0
\end{eqnarray}
\end{ex}
Since the system (\ref{dnny8})-(\ref{dnyy8}) has a nontrivial
solution for $A$ and $B$ the characteristic determinant is zero i.e.
$$a(\lambda)=\left|
    \begin{array}{cc}
      a_{11}(\lambda) & a_{12}(\lambda) \\
      a_{21}(\lambda) & a_{22}(\lambda) \\
    \end{array}
  \right|=0
$$
where
$a_{11}=1+((\alpha-1)^{2}+2(\alpha-1))\frac{(-\lambda)}{2!}+((\alpha-1)^{4}+4(\alpha-1)^{3})\frac{\lambda^{2}}{4!}+
((\alpha-1)^{6}+6(\alpha-1)^{5})\frac{(-\lambda^{3})}{6!}$,
$a_{12}=\alpha+((\alpha-1)^{3}+3(\alpha-1)^{2})\frac{(-\lambda)}{3!}\nonumber\\
+((\alpha-1)^{5}+5(\alpha-1)^{4})\frac{\lambda^{2}}{5!}$,
$a_{21}=1+(\alpha^{2}-2\alpha)\frac{(-\lambda)}{2!}+(\alpha^{4}-4\alpha^{3})\frac{\lambda^{2}}{4!}+
(\alpha^{6}-6\alpha^{5})\frac{(-\lambda^{3})}{6!}$,
$a_{22}=(\alpha-1)+(\alpha^{3}-3\alpha^{2})\frac{(-\lambda)}{3!}+(\alpha^{5}-5\alpha^{4})\frac{\lambda^{2}}{5!}$.\\
Taking $\alpha=\frac{1}{2}$ we have the following algebraic equation
for approximate  eigenvalues:
\begin{eqnarray}\label{dnmk}
-1-\frac{\lambda}{6}+\frac{11\lambda^{2}}{120}-\frac{89\lambda^{3}}{15360}+\frac{299\lambda^{4}}{2211840}-\frac{11\lambda^{5}}{9830400}=0
\end{eqnarray}
This equation can be solved by various numerical methods.

\section{Analysis of the method}
In this study we introduce a new version of classical DTM that will
extend the application of the method to spectral analysis of
boundary-value problems involving eigenvalue parameter, which  arise
from problems of mathematical physics. Numerical results reveal that
the $\alpha$-p DTM is a powerful tool for solving many initial value
and boundary value problems. It is concluded that comparing with the
standard DTM, the $\alpha$-p DTM reduces computational cost in
obtaining approximated solutions. This method unlike most numerical
techniques provides a closed-form solution. It may be concluded that
$\alpha$-p DTM is very powerful and efficient in finding approximate
for wide classes of  boundary value problems.  The main advantage of
the method is the fact that it provides its user with an analytical
approximation, in many cases an exact solution, in a rapidly
convergent sequence with elegantly computed terms.

\end{document}